\documentclass[aip,reprint,cha]{revtex4-2}
\usepackage{color}
\usepackage{svg}
\usepackage{amsmath}
\usepackage{amssymb}
\usepackage{amsthm}
\usepackage{bm}

\newtheorem{thm}{\protect\theoremname}
\theoremstyle{definition}
\newtheorem{defn}[thm]{Definition}
\theoremstyle{plain}
\newtheorem{prop}[thm]{Proposition}
\theoremstyle{plain}
\newtheorem{cor}[thm]{Corollary}
\theoremstyle{plain}
\newtheorem{lem}[thm]{Lemma}
\theoremstyle{remark}

\theoremstyle{remark}
\newtheorem{hypothesis}{Hypothesis}

\newcommand{\Tr}{\mathrm{Tr}}
\renewcommand{\Re}{\mathrm{Re}}
\renewcommand{\Im}{\mathrm{Im}}

\usepackage{mathtools}
\mathtoolsset{showonlyrefs}

\begin{document}
	
\title{Generalized splay states in phase oscillator networks} %Title of paper

\author{Rico Berner}
\email[]{rico.berner@physik.tu-berlin.de}
\affiliation{Institute of Theoretical Physics, Technische Universit\"at Berlin, Hardenbergstr. 36, 10623 Berlin, Germany}
\affiliation{Institute of Mathematics, Technische Universit\"at Berlin, Strasse des 17. Juni 136, 10623 Berlin, Germany.}
\author{Serhiy Yanchuk}
\affiliation{Institute of Mathematics, Technische Universit\"at Berlin, Strasse des 17. Juni 136, 10623 Berlin, Germany.}
\author{Yuri Maistrenko}
\affiliation{Forschungszentrum Jülich GmbH, Wilhelm-Johnen-Straße, 52428 Jülich, Germany}
\affiliation{Institute of Mathematics and Centre for Medical and Biotechnical Research, NAS of Ukraine, Tereshchenkivska St. 3, 01601 Kyiv, Ukraine.}
\author{Eckehard Sch\"oll}
\affiliation{Institute of Theoretical Physics, Technische Universit\"at Berlin, Hardenbergstr. 36, 10623 Berlin, Germany}
\affiliation{Bernstein Center for Computational Neuroscience Berlin, Humboldt Universit\"at, Philippstraße 13, 10115 Berlin, Germany}
\affiliation{Potsdam Institute for Climate Impact Research, Telegrafenberg A 31, 14473 Potsdam, Germany}

\date{\today}

\begin{abstract}
	Networks of coupled phase oscillators play an important role in the analysis of emergent collective phenomena. In this article, we introduce generalized $m$-splay states constituting a special subclass of phase-locked states with vanishing $m$th order parameter. Such states typically manifest incoherent dynamics, and they often create high-dimensional families of solutions (splay manifolds). For a general class of phase oscillator networks, we provide explicit linear stability conditions for splay states and exemplify our results with the well-known Kuramoto-Sakaguchi model. Importantly, our stability conditions are expressed in terms of just a few observables such as the order parameter or the trace of the Jacobian. As a result, these conditions are simple and applicable to networks of arbitrary size. We generalize our findings to phase oscillators with inertia and adaptively coupled phase oscillator models. %For the Kuramoto-Sakaguchi model with inertia, we show that the elements of the splay manifold can have different stability properties. 
\end{abstract}


\maketitle %\maketitle must follow title, authors, abstract and \pacs

\begin{quotation}
	Models of coupled phase oscillators are well-known paradigmatic systems to understand the mechanism behind the emergence of collective phenomena in complex networks. Due to the relative simplicity of these models, powerful methods such as the Watanabe-Strogatz theory or the Ott-Antonsen approach have been developed to describe certain dynamical states. Nowadays, a plethora of generalizations of phase oscillator models are developed to study biological, technological or socio-economic systems. Of particular interest are phase models with inertia for mechanical rotors in power grids or models with an adaptive network structure to model synaptic plasticity mechanisms in neuronal systems. This paper provides a systematic study of generalized splay states as a particular class of incoherent phase-locked solutions playing an important role in shaping the global dynamics of coupled oscillator systems.
In particular, we describe when a continuum of splay states emerges and a part of it (also continuum) becomes stable. These splay states are a manifestation of individual variability as one of the inherent properties of oscillatory networks.  
%Our findings are further generalized for phase oscillator models with inertia and generic adaptive degrees of freedom.		
\end{quotation}
%----------------------------------------
% section: Introduction
%----------------------------------------
%\the\columnwidth
\section{Introduction}\label{sec:intro}
Dynamical networks of phase oscillators are a well-known paradigm for studying the collective behavior of interacting agents~\cite{ACE05,PIK01}. 
The importance of such network models relies, in particular, on the fact that any system of weakly interacting nonlinear oscillators can be generally reduced to a phase oscillator network~\cite{WIN80,HOP97,PIK01,PIE19a}.
Extensive reviews have highlighted the importance of phase oscillator models and reduction techniques~\cite{PIK01,ASH16}. Recent studies also aim at increasing the range of applicability of phase oscillators by generalizing the conditions under which reduction techniques are valid~\cite{KLI17a,ROS19a,ERM19}.

A famous representative of the class of phase oscillator models, the Kuramoto model, in which all oscillators are coupled in the ''all-to-all'' manner, has attracted much attention due to its simple form and mathematical tractability~\cite{KUR84,STR00}. 
The Kuramoto model and its extensions have gained additional popularity through applications to real-world problems~\cite{STR93,PIK01,STR03,ROD16}, including
neuroscience~\cite{BRE10h,LUE16,ASL18a,ROE19a,BER21b,BIC20} and power grids~\cite{FIL08a,TUM19,TAH19,HEL20,TOT20,BER21a}. 
Despite the simple structure, the Kuramoto model can exhibit many different dynamical regimes~\cite{MAI14a,OME19c,TEI19}, and sophisticated methods have been developed for their analysis. In particular, it was shown that sinusoidally and globally coupled phase oscillators are partially integrable. The Watanabe-Strogatz theory allows for a reduction to only three dimensions, which can also be applied to even more general classes of phase oscillator models~\cite{WAT93a,WAT94,MAR09c,STE11b}. A remarkable observation in the work of Watanabe and Strogatz is the role of incoherent states~\cite{STR91} in the foliation of the phase space. They have shown that sheets of the foliation can be parametrized by the family of incoherent states. In our work, we analyze these incoherent but phase-locked and frequency-synchronized states for a large class of phase oscillator models and shed new light on their dynamical properties. Moreover, we generalize the notion of incoherent states by introducing generalized splay states.

Another approach developed to understand coupled oscillators in the continuum limit is the Ott-Antonson ansatz~\cite{OTT08}. In the case of an infinite number of oscillators, the Ott-Antonson theory allows for a reduction to a two-dimensional dynamical system and has been successfully applied to describe the emergence of partially synchronized patterns~\cite{OME08,ABR08,OME13,OME18a,OME19c}. Remarkably, for both reduction techniques, the Watanabe-Strogatz and the Ott-Antonson theory, the reduced systems possess a clear physical interpretation, and both approaches are closely related~\cite{MAR09c,PIK15}. In fact, this direct relationship between the two approaches makes the study of incoherent states (generalized splay states) important for mean-field and other reduction techniques~\cite{GOT15,GON19a,SMI20a} as well as for the search for future generalizations~\cite{PIK08,MON15,TYU18,GOL21,RON21}. Moreover, splay states play a role in networks with non-global coupling. These state have also been found in non-locally coupled ring networks~\cite{WIL06,GIR12,BUR18}, and the concept of local splay states (or local incoherent states) has been introduced to relate them to the the incoherent states in globally coupled networks~\cite{BER20c}. Furthermore, splay and incoherent states have been discussed for more complex coupled systems such as Stuart-Landau oscillators~\cite{CHO09,ZOU09b}, systems with delay~\cite{CHO09,PER10c} and pulse-coupling~\cite{CAL09a,OLM14} where for the latter a link to the Watanabe-Strogatz theory has been proved~\cite{DIP12}. 

Various generalization have been proposed beyond the classical Kuramoto model. Starting from the generalization to complex networks~\cite{GOM07,DOE14}, the theory of phase oscillators has been further developed to study phenomena of phase transitions~\cite{PAZ05a,GOM11a,BOC16}, network symmetries~\cite{NIC13}, the impact of inertia~\cite{ERM91,FIL08a,MAI14a,OLM14a,JAR15,OLM15a,BEL16a,MAI17,JAR18,KRU20a,BRE21} and other forms of frequency adaptation~\cite{TAY10}, delayed coupling~\cite{YEU99a}, or the effect of time-dependent parameters~\cite{PET12}, to name just a few.

Another generalization which has gained much attention in recent years concerns the phenomena in networks of phase oscillators with adaptive coupling. Several models have been proposed and studied to gain insights into the interplay between collective dynamics and adaptivity~\cite{SEL02,MAI07,AOK12,LUE16,KAS17,BAC18b,BER19,ROE19a,FRA20,BER20,BER21b,VOC21}. Many of them were inspired by recent findings in neuroscience related to synaptic plasticity.

In this work, we provide a general analytic study of the local properties, existence and stability, of incoherent phase-locked states. For this, we introduce the class of phase oscillators models in Section~\ref{sec:model} and define the notion of generalized $m$-splay state. In Section~\ref{sec:1Cl_stab}, we describe manifolds of the splay states and provide explicit conditions for their linear stability. In the subsequent sections~\ref{sec:POwI} and~\ref{sec:POadap}, we generalize the results for the stability of $m$-splay states to phase oscillator models with inertia and adaptivity, respectively. In Section~\ref{sec:geometrical} we give a geometrical perspective. The results are discussed in Section~\ref{sec:conclusion}. To make the main text more accessible, we have moved some proofs to the Appendix.

%----------------------------------------
% section: Generalized splay states and the model of coupled phase oscillator
%----------------------------------------
\section{Coupled phase oscillator models and generalized splay states}\label{sec:model}

We consider systems of $N$ coupled phase oscillators
\begin{align}\label{eq:phaseoscModel_general}
\frac{\mathrm{d}}{\mathrm{d}t}\bm{\phi} & = \omega \bm{1} + F(\bm{\phi}),
\end{align}
 where $\bm{\phi}=(\phi_1,\dots,\phi_N)^T$, and each oscillator is represented by a dynamical variable $\phi_i(t)\in[0,2\pi)$, $i=1,\dots,N$. 
All oscillators possess the same common natural frequency $\omega$. Moreover,  $F=(f_1(\bm{\phi}),\dots,f_N(\bm{\phi}))^T$ is the coupling vector field with coupling functions $f_i$, and $\bm{1}=(1,\dots,1)^T$.

To measure the phase coherence, we define the $m$th moment of the complex mean field, $m \in \mathbb N$, as
\begin{align}
	Z_m(\bm{\phi}) = \frac{1}{N}\sum_{j=1}^{N}e^{\mathrm{i}m\phi_j} = R_m(\bm{\phi})e^{\mathrm{i}\rho_m(\bm{\phi})},
\end{align}
where $\mathrm{i}$ is the imaginary unit, $R_m$ denotes the $m$th moment of the (Kuramoto-Daido) order parameter, and $\rho_m$ is the collective phase of the $m$th moment of the mean field~\cite{KUR84,DAI94}.

\begin{defn}\label{def:phaseLockedStates}
	A solution of the phase oscillator system~\eqref{eq:phaseoscModel_general} is called  \textit{phase-locked state} if
	\begin{align}\label{eq:phaseLockedStates}
		\phi_i(t) = \Omega t + \vartheta_i,\quad i=1,\dots,N,
	\end{align}
	with collective frequency  $\Omega\in\mathbb{R}$ and fixed relative phases $\vartheta_i\in [0,2\pi)$ of the individual oscillators.
\end{defn}

Phase-locked states for oscillator models have been studied extensively in the past, see e.g. Refs.~\onlinecite{MAI07,BER19a}. In this paper, we restrict our attention to a special subclass of phase-locked states for which little is known about their role in the case of finite ensembles of oscillators, as follows.
\begin{defn}\label{def:generalizedSplayStates}
	A phase-locked state with $\phi_i(t) = \Omega t + \vartheta_i$ is an \textit{$m$-splay state} if it satisfies  
	\begin{equation}
		\label{eq:msplaycondition}
		R_m(\bm{\vartheta})=0.
	\end{equation}
	 We call $m\in\mathbb{N}$ the \textit{moment} of the splay state, and Eq.~(\ref{eq:msplaycondition}) the \textit{$m$-splay condition}.
\end{defn}

This definition generalizes the "classical" notion of splay states~\cite{CHO09} that are defined by phase distributions with equidistant phase relation $\vartheta_j = k j 2\pi/N$ with $k=0,\dots,N-1$ and form an  $m$-splay state if $(mk \mod N) \ne 0$. These states are also referred to as twisted states~\cite{WIL06,GIR12,OME14} or rotating waves~\cite{STE03,BUR18}, and are often related to certain network symmetries~\cite{ASH92b,ASH08,ASH16a}. In Fig.~\ref{fig:SplayExamples}, we illustrate $1$- and $2$-splay states for ensembles of $N=2,3$, and $4$ oscillators. The relation $R_m(\bm{\vartheta})=R_1(m\bm{\vartheta})$ holds between the splay states with different moments. In particular, we observe that any $2$-splay state in Fig.~\ref{fig:SplayExamples}(d)-(f) corresponds to a $1$-splay state in Fig.~\ref{fig:SplayExamples}(a)-(c) by doubling the relative angles between the oscillators.

\begin{figure}
	\centering
	\includegraphics{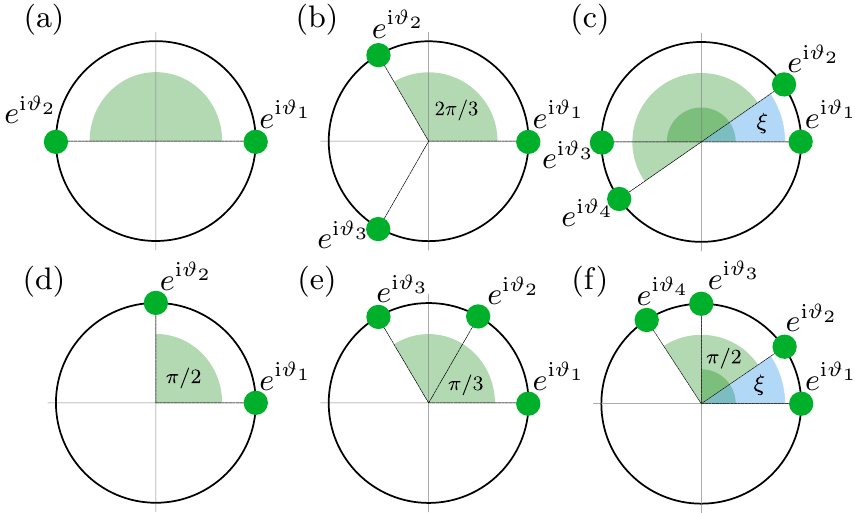}
	\caption{Illustrations of (a),(b),(c) $1$- and (d),(e),(f) $2$-splay states for $N=2$ (left column), $N=3$ (center column), and $N=4$ (right column) coupled phase oscillators represented on the unit circle. The green and blue angles depict fixed and parametrized (variable) phase relations, respectively.
		\label{fig:SplayExamples}
}
\end{figure}

With the definition \ref{def:generalizedSplayStates} of splay states, we may consider a whole family that fulfill the $m$-splay condition \eqref{eq:msplaycondition}.
\begin{defn}
	\label{def:splayStateFamily}
	The set
	\begin{align}\label{eq:splayStateFamily}
		SM_m = \left\{\bm{\vartheta}\in [0,2\pi)^N: R_m(\bm{\vartheta})=0\right\}
	\end{align} 
	is called the \textit{$m$-splay manifold}.
\end{defn}
The fact that $SM_m$ forms a $N-2$ dimensional manifold has been proven in Ref.~\onlinecite{ASH08}. Note that the splay manifold has a shift symmetry, i.e., if $\bm{\vartheta}\in SM_m$ then $\bm{\vartheta}+\psi\bm{1} \in SM_m$ for any $\psi\in [0,2\pi)$.

In this paper, we derive linear stability conditions for $m$-splay states for generic phase oscillator models \eqref{eq:phaseoscModel_general} that possess phase-shift symmetry and leave the $m$-splay manifold invariant, and each point of this manifold corresponds to an $m$-splay solution. 

More specifically, we assume that for some $m$, the following hypotheses are fulfilled:
\begin{hypothesis}\label{hyp:mSplayExists}
	For all $\bm{\vartheta}\in SM_m$,
	the system of coupled phase oscillators possesses $m$-splay states $\bm{\phi}(t)=\Omega t +\bm{\vartheta}$ with collective frequency $\Omega$.
\end{hypothesis}
\begin{hypothesis}\label{hyp:phaseinv}
For any $\psi\in \mathbb R$, the nonlinearity $F$ satisfies
$F(\bm{\phi}+\psi\bm{1}) = F(\bm{\phi})$. This implies that the corresponding system \eqref{eq:phaseoscModel_general} is equivariant with respect to the phase-shift transformation. 
\end{hypothesis}
With both Hypotheses~\ref{hyp:mSplayExists} and~\ref{hyp:phaseinv}, we guarantee that, first, all elements of the splay manifold $SM_m$ describe a phase-locked state and, second, the phase-locked states may be considered as time independent due to the phase-shift symmetry, i.e., we can consider the co-rotating reference frame $\bm{\phi}\to \bm{\phi}+\Omega t$. These restrictions are met by many phase oscillator models that have been analyzed over the last decades, including the Kuramoto-Sakaguchi model~\cite{SAK86}, models with higher mode coupling~\cite{DEL19} and with higher order interactions~\cite{SKA20}, models of coupled phase oscillators under resource constraints~\cite{KRO20b} and generalized phase oscillator models including systems with inertia~\cite{OLM14a,JAR18} or adaptive network structure~\cite{KAS17}. Some of them are discussed in the subsequent sections. An important class of systems that fulfill Hypothesis \ref{hyp:mSplayExists} are those that are coupled via mean-field \cite{PIK15,BIC20}. %or via phase difference variables. 
%We think that the above Hypothesis may be lifted in future work.

In the next section, we derive conditions for the linear stability of the $m$-splay manifold of system~\eqref{eq:phaseoscModel_general} under the Hypothesis'~\ref{hyp:mSplayExists} and~\ref{hyp:phaseinv}.
%----------------------------------------
% section: Stability of generalized splay clusters in pure phase oscillator models
%----------------------------------------

\section{Stability of generalized splay states in phase oscillator models}\label{sec:1Cl_stab}
This section explores the linear stability of $m$-splay states as defined in the previous section. In the first part, we provide a general result for the generic class of phase oscillator models. This result is then discussed for the Kuramoto-Sakaguchi model.
%----------------------------------------
% subsection: Applictaions to models of particaluar systems
%----------------------------------------
\subsection{General result on the stability of splay states}\label{sec:splayStability}
We start with the variational equation around an arbitrary $m$-splay state of system \eqref{eq:phaseoscModel_general}, which satisfies Hypetheses~\ref{hyp:mSplayExists} and \ref{hyp:phaseinv}.
This variational equation reads
\begin{align}\label{eq:VariEq}
	\frac{\mathrm{d}\delta\bm{\phi}(t)}{\mathrm{d}t}= L(\bm{\vartheta})	\delta\bm{\phi},
\end{align}
where $L(\bm{\vartheta})=\mathrm{D}F$ denotes the Jacobian of the coupling field $F$. We note that due to the shift symmetry of system \eqref{eq:phaseoscModel_general}, the Jacobian is time-independent, and has zero row sum for each row and thus possesses a zero eigenvalue corresponding to the eigenvector $\bm{1}$, i.e., $L(\bm{\vartheta})\bm{1} = 0$, for any $\bm{\vartheta}$. This eigenvector $\bm{1}$ acts along the symmetry action, see Hypothesis~\ref{hyp:phaseinv}.
 The nondiagonal entries of the $N\times N$ matrix $L$ are 
\begin{align}\label{eq:splayJacobianEntriesNonDiag}
	l_{ij} = \frac{\partial f_i}{\partial \phi_j}(\bm{\vartheta}).
\end{align}
Due to the zero row-sum condition, the diagonal elements $l_{ii}$ are given as
\begin{align}\label{eq:splayJacobianEntriesDiag}
	l_{ii}= -\sum_{j=1,j\ne i}^N\frac{\partial f_i}{\partial \phi_j}(\bm{\vartheta}).
\end{align}

The linear stability of the $m$-splay states is determined by the eigenvalues of the Jacobian matrix $L$. More precisely, we are interested in the real parts of these eigenvalues. The following Lemma provides useful insights into the spectral structure of a special class of matrices $L$ that are of major importance subsequently.

We denote a polynomial of degree $r\in\mathbb{N}_0$ over the complex field $\mathbb{C}$ with complex argument $\lambda\in\mathbb{C}$ as $p_r(\lambda)$, i.e., $p_r(\lambda)=\sum_{k=0}^r a_k \lambda^k$. The characteristic polynomial of an $N\times N$ matrix is denoted by $p_N(L,\lambda)$, i.e., $p_N(L,\lambda) = \det(L-\lambda \mathbb{I}_N)$.

\begin{lem}\label{lem:splayMatrixCharPolGeneral}
	Suppose an $N\times N$ ($N>1$) matrix $L$  possesses a zero eigenvalue with multiplicity $N-2$. Then, the characteristic polynomial $p_N(L,\lambda)$ is given by
	\begin{align}
		\det\left(L-\lambda\mathbb{I}_N\right) = (-1)^N\lambda^{(N-2)}\left(\lambda^2+a_{(N-1)}\lambda+a_{(N-2)}\right)
	\end{align}
	with the coefficients
	\begin{align}
		a_{(N-1)} &= -\sum_{j=1}^N l_{jj} = -\mathrm{Tr}(L),\\
		a_{(N-2)} &= \sum_{i=i}^N \sum_{j>i}^N \left(l_{ii}l_{jj}-l_{ij}l_{ji}\right) = 
		\frac{1}{2} \left( \mathrm{Tr}(L)^2 - \mathrm{Tr}(L^2) \right).  %\frac{1}{2}\sum_{i,j=1}^N \left(l_{ii}l_{jj}-l_{ij}l_{ji}\right).
	\end{align}
\end{lem}
This results is a consequence of a general theorem on the coefficients of the characteristic polynomial~\cite{HOU98a}. However, we provide a direct proof of Lemma~\ref{lem:splayMatrixCharPolGeneral} in the Appendix~\ref{sec:proofSplayLem} for the interested reader. With Lemma~\ref{lem:splayMatrixCharPolGeneral}, we are able to write explicit stability conditions for any $m$-splay state depending on the two explicit characteristics of the Jacobian $L$: its trace $\mathrm{Tr}\,(L)$ and the trace of its square $\mathrm{Tr}\,(L^2)$.

\begin{prop}\label{prop:splayStateLinStability}
	Suppose $L(\bm{\vartheta})$ is the Jacobian from~\eqref{eq:VariEq}, whose entries are given in~\eqref{eq:splayJacobianEntriesNonDiag} and~\eqref{eq:splayJacobianEntriesDiag}, with $\bm{\vartheta}$ corresponding to an $m$-splay state  of the coupled phase oscillator system~\eqref{eq:phaseoscModel_general}. Then, we distinguish the following two cases:\\
	(i) If  $2\mathrm{Tr}(L^2)\le\mathrm{Tr}(L)^2$, then the $m$-splay state is linearly stable if and only if
	\begin{align}
		\mathrm{Tr}(L)<0.
	\end{align} 
	(ii) If $2\mathrm{Tr}(L^2)>\mathrm{Tr}(L)^2$, then the $m$-splay state is linearly stable if and only if
	\begin{align}
		\mathrm{Tr}(L)<0  \quad \mbox{and}\quad 				\mathrm{Tr}(L^2)<\mathrm{Tr}(L)^2.
	\end{align}
\end{prop}
\begin{proof}
	Since $\bm{\vartheta}$ corresponds to an $m$-splay state, there are $N-2$ neutral perturbation directions along the splay manifold. These perturbations are determined by the condition
	\begin{align*}
		\sum_{j=1}^N e^{\mathrm{i}m\vartheta_j}\delta\phi_j = 0,
	\end{align*}
	which follows from $\delta Z_m(\bm{\phi})=0$.
	In particular, the perturbation $\delta \bm{\phi}=\bm{1}$ 
	constitutes one of the dimensions of the neutral subspace.
	Thus $L$ possesses a zero eigenvalue with multiplicity $N-2$, and we can apply Lemma~\ref{lem:splayMatrixCharPolGeneral}. We obtain a quadratic equation with real-valued coefficients depending on $\mathrm{Tr}(L)$ and $\mathrm{Tr}(L^2)$ whose solution is given by
	\begin{align}
	\label{eq:Lambda}
		\lambda_{1,2} = \frac{\mathrm{Tr}(L)}{2} \pm \frac12 \sqrt{2\mathrm{Tr}(L^2)-\mathrm{Tr}(L)^2}.
	\end{align}
	The two cases follow immediately by considering the real parts of the solutions $\lambda_{1,2}$.
\end{proof}

Proposition~\ref{prop:splayStateLinStability} provides a very general linear stability condition that depends on two features of the Jacobian $L$ only. In particular, it depends on the sum of all eigenvalues $\lambda_i$ of $L$, i.e, $\Tr(L)=\sum_{i=1}^N \lambda_i$ and on the sum of all squares of eigenvalues, i.e., $\Tr(L^2)=\sum_{i=1}^N \lambda^2_i$.
\begin{figure}
	\centering
	\includegraphics[width=\columnwidth]{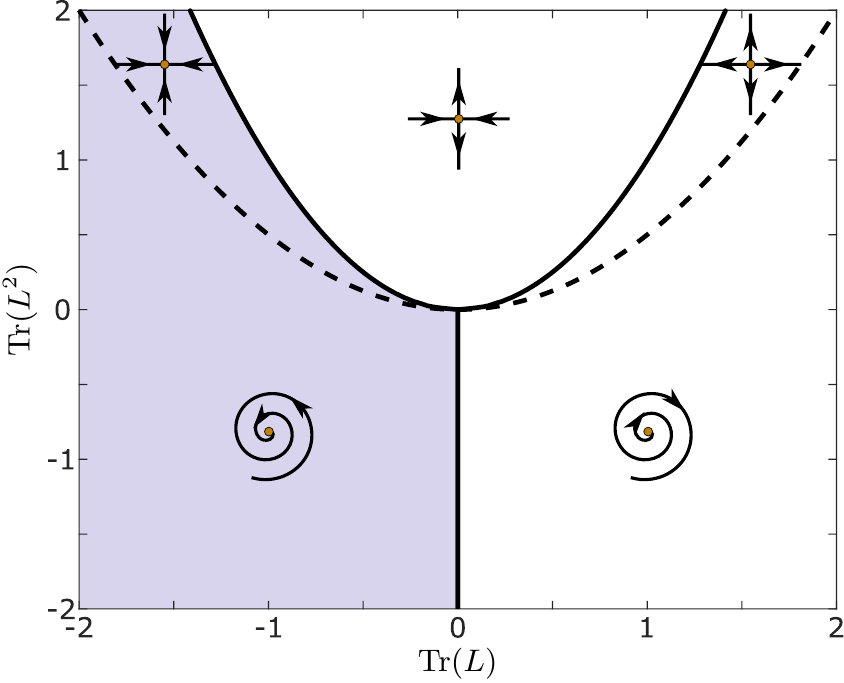}
	\caption{
		\label{fig:PhaseDiagramPO}
		Diagram showing the local properties of $m$-splay states in dependence of the values $\Tr(L)$ and $\Tr(L^2)$. The solid black curves indicate transitions between different stability features of the $m$-splay state, i.e., stable, saddle, repelling. The dashed lines indicate transition between nodes and foci. The shaded parts of the diagram correspond to linear stability.}
\end{figure}
Figure~\ref{fig:PhaseDiagramPO} shows the stability region for an arbitrary $m$-splay state in the  $(\Tr(L),\Tr(L^2))$-plane. In particular, the solid dashed lines indicate transition for which the real part of at least one eigenvalue crosses zero. An analytic expressions for theses lines is derived as follows. Assume that one eigenvalue $\lambda$ possesses zero real part, i.e., $\lambda=\mathrm{i}v$ with $v\in\mathbb{R}$. Substituting this assumption into the quadratic expression of Lemma~\ref{lem:splayMatrixCharPolGeneral}, we obtain
\begin{align*}
	-v^2 -\mathrm{i}\Tr(L)v + \frac{1}{2}(\Tr(L)^2-\Tr(L^2)) = 0.
\end{align*}
The latter equation can either be fulfilled with $v=0$ and $\Tr(L^2)=\Tr(L)^2$ or with $\Tr(L)=0$ and $\Tr(L^2)=-v^2$ for all $v\in\mathbb{R}$. Both conditions agree with the solid black lines in Figure~\ref{fig:PhaseDiagramPO}.

In the following sections, we apply the results of Lemma~\ref{lem:splayMatrixCharPolGeneral} and Proposition~\ref{prop:splayStateLinStability} in order to describe the stability of splay states for a particular model.
%----------------------------------------
% section: Applications to models of particaluar kind
%----------------------------------------
\subsection{Kuramoto-Sakaguchi model}\label{sec:KSmodel}
In this section, we study the linear stability of splay states in a globally coupled network of Kuramoto-Sakaguchi phase oscillators~\cite{KUR84,SAK86} given by
\begin{align}\label{eq:KSmodel}
	\dot{\phi}_i = \omega - \frac{1}{N}\sum_{j=1}^N \sin(\phi_i - \phi_j + \alpha),
\end{align}
where $\alpha$ is the phase-lag parameter. 
This system satisfies the Hypothesis~\ref{hyp:phaseinv} of phase-shift invariance. 
We note that the coupling function of~\eqref{eq:KSmodel} and its derivative can be written as
\begin{align}
	f_i &= -\Im\left(\overline{Z}_1e^{\mathrm{i}\alpha}e^{\mathrm{i}\phi_i}\right),\\
	\frac{\partial f_i}{\partial \phi_j} &= \frac{1}{N}\Re\left(e^{\mathrm{i}(\phi_i-\phi_j+\alpha)}\right).
\end{align}
With this, we immediately see that any phase distribution $\bm{\vartheta}$ that fulfills the $1$-splay condition $Z_1(\bm{\vartheta})=0$, corresponds to the 1-splay solution $\phi_i(t)=\omega t + \vartheta_i$ of \eqref{eq:KSmodel}. 
Therefore, system \eqref{eq:KSmodel} satisfies the Hypothesis~\ref{hyp:mSplayExists}. 
To derive the stability condition with Proposition~\ref{prop:splayStateLinStability}, we determine the entries of the Jacobian matrix $L(\bm{\vartheta})$ of the variational system for~\eqref{eq:KSmodel} around the $1$-splay states. The entries read
\begin{align}\label{eq:splayJacobianEntries_KSmodel}
	l_{ij}=\begin{cases}
		\frac{1}{N}\cos(\alpha)& \text{if }j=i, \\
		\frac{1}{N}\cos(\vartheta_i-\vartheta_j+\alpha) & \text{otherwise}. \\
	\end{cases}
\end{align}

With these preliminaries we obtain the following.
\begin{cor}\label{cor:splayStateLinStability_KSmodel}
	The $1$-splay state $\bm{\phi}=\omega \bm{1} + \bm{\vartheta}$ of Kuramoto-Sakaguchi system \eqref{eq:KSmodel} is linearly stable if and only if
	\begin{align}
		\cos\alpha<0.
	\end{align}
\end{cor}
\begin{proof}
	In order to prove this result we use  Lemma~\ref{lem:splayMatrixCharPolGeneral} and determine the coefficients of the quadratic polynomial. We obtain the following:
	\begin{align}
		 a_{(N-1)} = -\mathrm{Tr}(L) = -\cos\alpha
	\end{align}
	and
	\begin{align*}
		a_{(N-2)} &= \frac{\cos^2(\alpha)}{2}  - \frac{1}{4N^2}\sum_{i,j=1}^N \left(\cos(2(\vartheta_i-\vartheta_j))+\cos(2\alpha)\right) \\
		&= \frac{1}{4} - \frac{1}{4N}\sum_{j=1}^N R_2(\bm{\vartheta})\Re\left(e^{\mathrm{i}(\vartheta_j-\rho_2(\bm{\vartheta}))}\right) \\
		&= \frac{1}{4} \left(1 - R^2_2(\bm{\vartheta})\right).
	\end{align*}
	Since $1 - R^2_2(\bm{\vartheta})\ge 0$ for any $\bm{\vartheta}$, the stability condition for the $1$-splay state is $a_{(N-1)}>0$, which is equivalent to 
	 $\cos(\alpha)<0$.
\end{proof}

Corollary~\ref{cor:splayStateLinStability_KSmodel} implies that a $1$-splay state is linearly stable for~\eqref{eq:KSmodel} as long as $\cos\alpha<0$. Thus, the stability does not depend on the particular shape of the splay state, i.e., the particular element $\bm{\vartheta}$ of the manifold $SM_1$. However, the bifurcation occurring at $\cos\alpha=0$ might be different depending on the distribution of the phases. 
%\myrem{SY}{for the bifurcation $\cos\alpha=0$, we have always imaginary parts (except for $R_2=1$), hence, it is always Hopf, right? 
%~~~~~~~~~~~~~YM:  This is not a regular Hopf. The bifurcation is highly degenerate of some co-dimention: a sharp transition to one of the huge number (continuum) of possible stable splay states growing from this bifurcation point. This is a kind of 'extreme sensitivity to initial conditions' or 'individual variability'? We can discuss  this property  (which I consider a generic for networks) or leave to future study}
Due to Proposition~\ref{prop:splayStateLinStability}, the eigenvalues of the Jacobian~\eqref{eq:splayJacobianEntries_KSmodel} may have imaginary parts if $2\mathrm{Tr}(L^2)<\mathrm{Tr}(L)^2$, i.e., $\cos^2\alpha<(1-R^2_2(\bm{\vartheta}))$. In this case, the imaginary parts are given by
% $\pm\sqrt{\cos^2\alpha-(1-R^2_2(\bm{\vartheta}))}/2=\pm\sqrt{R^2_2(\bm{\vartheta})-\sin^2\alpha}/2$. 
$$\mathrm{Im}\,\lambda_{1,2}=\pm\frac{1}{2}\sqrt{\sin^2\alpha - R^2_2(\bm{\vartheta})}.$$ 
We further note that the second moment of the order parameter $R_2$ characterizes the whole $1$-splay manifold with respect to the transverse stability.

%----------------------------------------
% section: Beyond pure phase oscillators
%----------------------------------------
%\section{Beyond pure phase oscillator models}\label{sec:applBeyondPhaseOscModel}
In the following sections, we extend the previous findings to more complex models of coupled phase oscillators. More precisely, we consider phase oscillator models with inertia and models of adaptively coupled phase oscillators.
%----------------------------------------
% subsection: Phase oscillator models with inertia
%----------------------------------------
\section{Phase oscillator models with inertia\label{sec:POwI}}

As a first extension of our results from Section~\ref{sec:1Cl_stab}, we study the linear stability of generalized splay states in models of $N$ coupled phase oscillators with inertia. Similar to the previous section, we first provide the general result and then consider a more specific class of models.
\subsection{Stability of $m$-splay states\label{sec:POwISplay}}
We consider the following class of phase oscillator models with inertia
\begin{align}\label{eq:POwI_2order}
	M\frac{\mathrm{d}^2}{\mathrm{d}t^2}\bm{\phi} +\gamma\frac{\mathrm{d}}{\mathrm{d}t}\bm{\phi} & =  p \bm{1} + F(\bm{\phi}),
\end{align}
where $p\in\mathbb{R}$ corresponds to the nondimensionalized power generation and consumption in power grid models and is related to the natural frequency in Eq.~\eqref{eq:phaseoscModel_general} by $\omega=p/\gamma$, the parameter $M$ is the inertia, and $\gamma>0$ is the damping constant, which extends Eq.~\eqref{eq:phaseoscModel_general}. Note that, for identical oscillators, the number of parameters can be reduced to two. Therefore, we assume $M=1$ in the following.

Here, we also assume that \eqref{eq:POwI_2order} possesses a set of $m$-splay states $\bm{\phi}=\Omega t + \bm{\vartheta}$ for all $\bm{\vartheta}\in SM_m$, i.e., Hypothesis~\ref{hyp:mSplayExists} holds. We may write \eqref{eq:POwI_2order} as a set of first order differential equations
\begin{align}\label{eq:POwI_1order}
	\frac{\mathrm{d}}{\mathrm{d}t}\bm{\phi} &= \bm{\psi},\\
	\frac{\mathrm{d}}{\mathrm{d}t}\bm{\psi} &= -\gamma\bm{\psi} +  p \bm{1} + F(\bm{\phi}),
\end{align}
where we introduce the new dynamical variable $\bm{\psi}\in\mathbb{R}^N$.

In order to determine the linear stability of $m$-splay states, we consider the variational equation for system~\eqref{eq:POwI_1order} around an arbitrary $m$-splay state, which reads
\begin{align}\label{eq:KwI_variEq}
	\frac{\mathrm{d}}{\mathrm{d}t}
	\begin{pmatrix}
		\delta\bm{\phi} \\
		\delta\bm{\psi}
	\end{pmatrix} =
	\begin{pmatrix}
		0 & \mathbb{I}_N \\
		L(\bm{\vartheta}) & -\gamma\mathbb{I}_N
	\end{pmatrix}
	\begin{pmatrix}
		\delta\bm{\phi} \\
		\delta\bm{\psi}
	\end{pmatrix}
	= J(\bm{\vartheta})
	\begin{pmatrix}
		\delta\bm{\phi} \\
		\delta\bm{\psi}
	\end{pmatrix}
\end{align}
where $J$ denotes the Jacobian and the entries of the $N\times N$ matrix $L$ are given as in~\eqref{eq:splayJacobianEntriesNonDiag}--\eqref{eq:splayJacobianEntriesDiag}.

The linear stability of the $m$-splay states is determined by the eigenvalues of the Jacobian matrix $J$. The following Lemma provides a useful tool to find these eigenvalues. 

\begin{lem}~\label{lem:KwISchurRed}
The $2N$ eigenvalues of the matrix $\begin{pmatrix}
		0 & m_1 \mathbb{I}_N \\
		L & m_2 \mathbb{I}_N
	\end{pmatrix}$  with $m_1,m_2\in\mathbb{C}$ are given by the solutions of the $N$ quadratic equations
	\begin{align}\label{eq:KwISchurRed_quadratic}
		\mu^2 - m_2\mu - m_1\lambda_i = 0, \quad i=1,\dots,N,
	\end{align}
	where $\lambda_1,\dots,\lambda_N$ are the eigenvalues of $L$.
\end{lem}
The proof of Lemma~\ref{lem:KwISchurRed} can be found in Ref.~\onlinecite{TUM19}. With this Lemma and Lemma~\ref{lem:splayMatrixCharPolGeneral}, the following conditions for the local stability of the $m$-splay states of \eqref{eq:POwI_2order} are derived.
\begin{prop}\label{prop:splayStateLinStability_POwImodel}
	Suppose $J(\bm{\vartheta})$ is the Jacobian of~\eqref{eq:KwI_variEq} and $L(\bm{\vartheta})$ possesses the entries as given in~\eqref{eq:splayJacobianEntriesNonDiag}--\eqref{eq:splayJacobianEntriesDiag}, where $\bm{\vartheta}$ corresponds to an $m$-splay state which solves~\eqref{eq:POwI_2order}. Then the $m$-splay state is linearly stable if and only if $\gamma>0$ and $\mathrm{Re}(\mu_{1,2,3,4})<0$, where 
%	\begin{align}\label{eq:POwI_SplayEigenVal01}
%		\mu_{1,2} = - \frac{\gamma}{2} \pm \frac{\sqrt{\gamma^2+2\left(\mathrm{Tr}(L) + \sqrt{2\mathrm{Tr}(L^2)-\mathrm{Tr}(L)^2}\right)}}{2},\\
%		\label{eq:POwI_SplayEigenVal02}
%		\mu_{3,4} = - \frac{\gamma}{2} \pm \frac{\sqrt{\gamma^2+2\left(\mathrm{Tr}(L) - \sqrt{2\mathrm{Tr}(L^2)-\mathrm{Tr}(L)^2}\right)}}{2},
%	\end{align}
	\begin{align}\label{eq:POwI_SplayEigenVal01}
	\mu_{1,2,3,4} = - \frac{\gamma}{2} \pm \sqrt{\left(\frac{\gamma}{2}\right)^2+ \lambda_{1,2}}, \\
	\lambda_{1,2} = \frac{\mathrm{Tr}(L)}{2} \pm \frac12 \sqrt{2\mathrm{Tr}(L^2)-\mathrm{Tr}(L)^2}.
	\end{align}
It is interesting to note that $\lambda_{1,2}$ equals the eigenvalues \eqref{eq:Lambda} of the phase oscillator model without inertia.
\end{prop}
\begin{proof}
	Due to Proposition~\ref{prop:splayStateLinStability}, the eigenvalues of $L$ are given by $\lambda_{3}=\dots=\lambda_N = 0$ and 
	\begin{align*}
		\lambda_{1,2}´ = \frac{\mathrm{Tr}(L)}{2} \pm \frac12 \sqrt{2\mathrm{Tr}(L^2)-\mathrm{Tr}(L)^2}.
	\end{align*}
	Using Lemma~\ref{lem:KwISchurRed}, the eigenvalues of the Jacobian $J$ are given by the solutions of 
	$$
	\mu^2 + \gamma \mu - \lambda_i = 0.
	$$
	The $N-2$ zero eigenvalues $\lambda_{3,\dots,N} = 0$ lead to $\mu=0$ and $\mu=-\gamma$, each with the multiplicity $N-2$.  Besides, there are roots 
	\begin{align}
		\mu_{i,1,2} = - \frac{\gamma}{2} \pm \sqrt{\left(\frac{\gamma}{2}\right)^2+\lambda_i}, \quad i=1,2,
	\end{align}
	which yield the result.
\end{proof}
Note that the characteristic polynomial of the Jacobian in~\eqref{eq:KwI_variEq} can be rewritten in a form that agrees with the findings in Ref.~\onlinecite{BRE21} for the case of $N=4$ oscillators. In particular, the eigenvalues~\eqref{eq:POwI_SplayEigenVal01} can be also expressed as solutions of the equation
\begin{multline}\label{eq:POwISplayStabQuartic}
	\mu^4 + 2\gamma\mu^3 + (\gamma^2-\mathrm{Tr}(L))\mu^2\\
	-\gamma\mathrm{Tr}(L)\mu+\frac{\mathrm{Tr}(L)^2-\mathrm{Tr}(L^2)}{2} = 0.
\end{multline}

We conclude that the stability properties of $m$-splay states depend only on the parameters $\gamma$, $\Tr(L)$ and $\Tr(L^2)$. 
The quantities $\Tr(L)$ and $\Tr(L^2)$ contain also information
about the specific splay states. The values $\Tr(L)$ and $\Tr(L^2)$
provide a foliation of the splay manifold so that each sheet of this foliation, with the same values of $\Tr(L)$ and $\Tr(L^2)$, possesses the same transverse local dynamics.
In order to find the boundary of the stability region, we substitute $\mu = \mathrm{i}v$ into \eqref{eq:POwISplayStabQuartic}, and find
\begin{multline}\label{eq:POwISurfaceEq}
	v^4 - (\gamma^2-\mathrm{Tr}(L))v^2 +\frac{\mathrm{Tr}(L)^2-\mathrm{Tr}(L^2)}{2}\\
	- \mathrm{i}\gamma v \left(2v^2+\mathrm{Tr}(L)\right) = 0.
\end{multline}
The latter equation is solved when either one of the following conditions is fulfilled: \\
(i) $v=0$ and $\mathrm{Tr}(L^2)=\mathrm{Tr}(L)^2$ for all $\gamma>0$, \\
(ii) $-2v^2=\Tr(L)$ and $\Tr(L)^2/2+\gamma \Tr(L) - \Tr(L^2)=0$ for all $v\in\mathbb{R}$. 

In Figure~\ref{fig:PhaseDiagramInertiaPO}(a), we display both conditions as surfaces in $(\Tr(L),\Tr(L^2),\gamma)$-space. In panels (b)-(e) of Fig.~\ref{fig:PhaseDiagramInertiaPO}, we show 2-parameter cross-sections with fixed values of $\gamma$, where stable regions are shaded. We note that the area corresponding to stable dynamics increases with increasing $\gamma$. 
\begin{figure}
	\centering
	\includegraphics{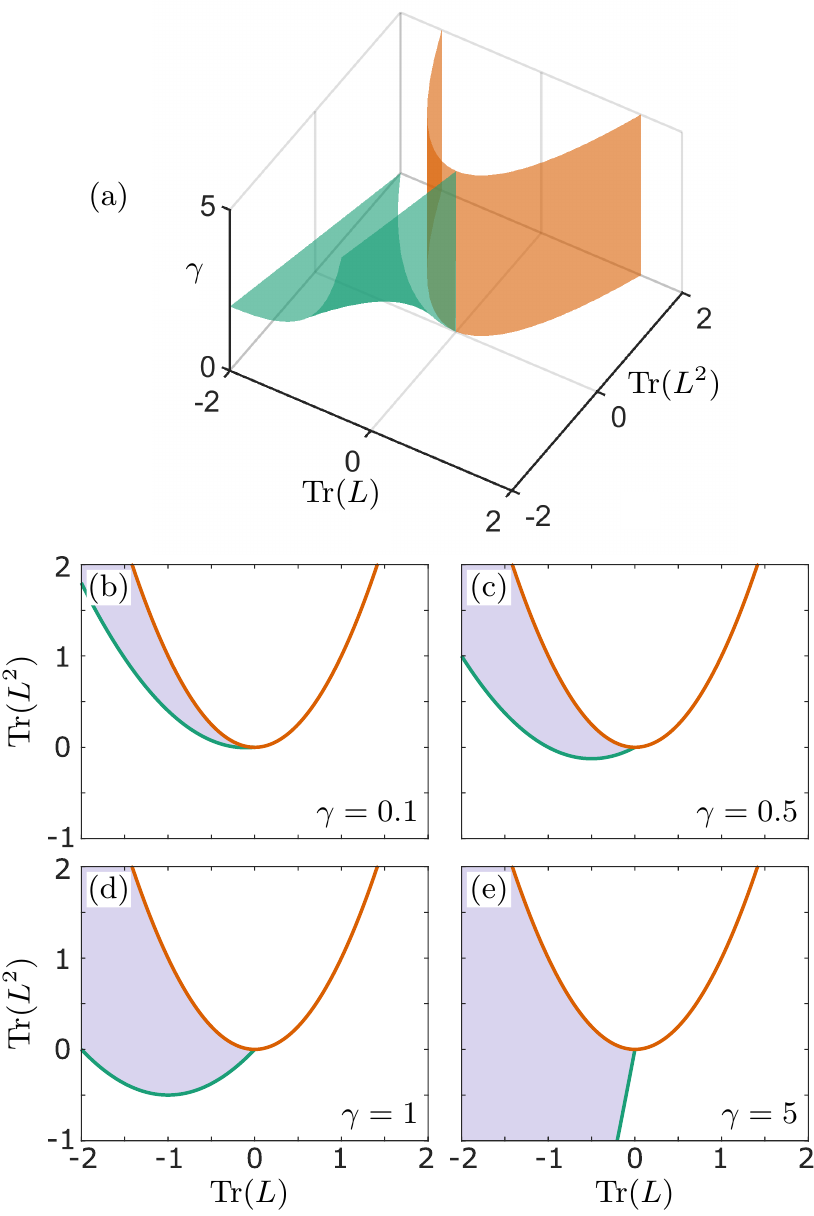}
	\caption{
		\label{fig:PhaseDiagramInertiaPO}
		Phase diagram showing the local stability of the $m$-splay states for system~\eqref{eq:POwI_2order} in dependence on $\gamma$, $\Tr(L)$, and $\Tr(L^2)$. In panel (a), the surfaces separating stable from unstable regimes are depicted in orange and green corresponding to the conditions (i) and (ii) for Eq.~\eqref{eq:POwISurfaceEq}, respectively. In panels (b)-(e), sections for fixed values  $\gamma=0.1,0.5,1,5$ are shown, respectively. Stable regions are shaded. The line colors indicate the cross-sections with the corresponding surfaces.}
\end{figure}

\subsection{Application to the Kuramoto-Sakaguchi model with inertia\label{sec:KSwI}}
In the following, we study the linear stability of generalized splay states in a globally coupled network of $N$ coupled phase oscillators with inertia~\cite{MAI17,JAR18,TAH19,TUM19,BER21a,TOT20,BRE21} of the form
\begin{align}\label{eq:KwI_2order}
	M\ddot{\phi}_i +\gamma\dot{\phi}_i & = p -\frac{\sigma}{N}\sum_{j=1}^N\sin(\phi_i-\phi_j+\alpha),
\end{align}
where $\phi_{i}\in[0,2\pi)$ represents the phase of the $i$th rotator. The parameter $M$ is the inertia constant, $\gamma>0$ is the damping constant, and $\sigma$ is the coupling constant. The parameter $\alpha$ can be regarded as a phase-lag of the interaction~\cite{SAK86}.

As for the Kuramoto-Sakaguchi model, see Section~\ref{sec:KSmodel}, the coupling functions and their derivatives can be written as
\begin{align}
	f_i &= -\sigma\Im\left(\overline{Z}_1e^{\mathrm{i}\alpha}e^{\mathrm{i}\phi_i}\right),\\
	\frac{\partial f_i}{\partial \phi_j} &= \frac{\sigma}{N}\Re\left(e^{\mathrm{i}(\phi_i-\phi_j+\alpha)}\right).
\end{align}
With this, we immediately see that any $1$-splay state is a solution of~\eqref{eq:KwI_2order} since $Z_1=0$ and hence $f_i=0$ for all $\bm{\vartheta}\in SM_1$ and $i=1,\dots,N$.

Proposition \ref{prop:splayStateLinStability_POwImodel} leads to the following criteria for the stability of 1-splay states.
\begin{cor}\label{cor:splayStateLinStability_KwImodel}
 The $1$-splay state of system \eqref{eq:KwI_2order} is linearly stable if and only if $\gamma>0$ and $\mathrm{Re}(\mu_{1,2,3,4})<0$, where
 	\begin{align}\label{eq:KwI_SplayEigenVal01}
 	\mu_{1,2,3,4} = - \frac{\gamma}{2} \pm \sqrt{\left(\frac{\gamma}{2}\right)^2+ \lambda_{1,2}}, \\
\lambda_{1,2} =
\frac{\sigma}{2}\left(\cos\alpha + \sqrt{R^2_2(\bm{\vartheta})-\sin^2\alpha}\right).
	\end{align}
\end{cor}
Note that the characteristic polynomial of the Jacobian in~\eqref{eq:KwI_variEq}  can be rewritten in a form that agrees with the findings in Ref.\onlinecite{BRE21} for the $N=4$ case. In particular, the eigenvalues~\eqref{eq:KwI_SplayEigenVal01} can be  expressed as solutions of the equation
\begin{multline}\label{eq:KwISplayStabQuartic}
	\mu^4 + 2\gamma\mu^3 + (\gamma^2-\sigma\cos(\alpha))\mu^2\\
	-\gamma\sigma\cos(\alpha)\mu+\frac{\sigma^2}{4}(1-R^2_2(\bm{\vartheta})) = 0.
\end{multline}

Similar to the general case above, we illustrate the stability properties of the $1$-splay state depending on $\gamma$, $\sigma$, $\alpha$, and $R_2$ that characterizes the splay manifold with respect to the linear stability. For this, note first that the parameters $\gamma$ and $\sigma$ can be reduced to one parameter $\gamma'=\sqrt{\sigma}\gamma$ with respect to the stability of the $1$-splay state. In particular, the mapping $\gamma\mapsto \sqrt{\sigma}\gamma$ and $\mu\mapsto\sqrt{\sigma}\mu$ leaves the stability features invariant but renders \eqref{eq:KwISplayStabQuartic} independent of $\sigma$. Hence, we may consider \eqref{eq:KwISplayStabQuartic} for $\sigma=1$ without loss of generality. 

As in previous cases, we look for the boundary of the stability region by substituting $\mu = \mathrm{i}v$ into \eqref{eq:KwISplayStabQuartic}:
\begin{multline}\label{eq:KwISurfaceEq}
	v^4 - (\gamma^2-\cos\alpha)v^2 +\frac{1}{4}(1-R^2_2(\bm{\vartheta}))\\
	- \mathrm{i}\gamma v \left(2v^2+\cos\alpha\right) = 0.
\end{multline}
The obtained equation is solved when either one of the following conditions is fulfilled:\\
(i) $v=0$ and $R^2_2(\bm{\vartheta})=1$ for all $\gamma>0$,
\\ (ii) $-2v^2=\cos\alpha$ and $\sin^2 \alpha - R^2_2(\bm{\vartheta}) + 2\gamma \cos\alpha=0$ for all $v\in\mathbb{R}$. \\
The first condition corresponds to a singular point on the splay manifold with $R_2=1$ and is not of general relevance. 
Hence, in Fig.~\ref{fig:R2AlphaDiagramKwI}(a), we display the second condition as a surface in $(\alpha,R^2_2(\bm{\vartheta}),\gamma)$-space. In  panels (b)-(e) of Fig.~\ref{fig:R2AlphaDiagramKwI}, we depict two-dimensional cross-section with fixed values of $\gamma$,where the stable regions are indicated by shading. We note that the area corresponding to stable dynamics increases for increasing $\gamma$. Moreover, we observe that for any $\gamma$ the stability intervals in $\alpha$ decrease from $[-\pi/2,\pi/2]$ for $R_2=1$ to a smaller interval for $R_2=0$, respectively.
\begin{figure}
	\centering
	\includegraphics[width=\columnwidth]{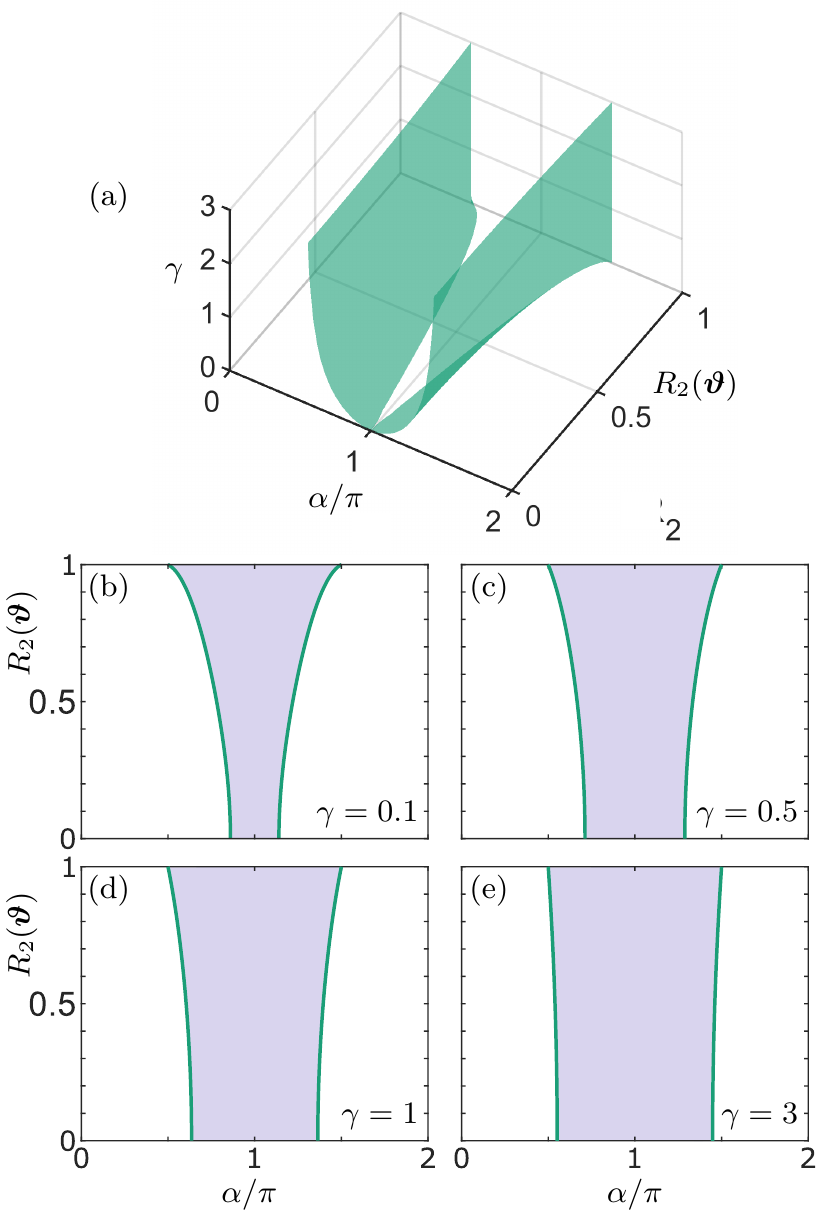}
	\caption{
		\label{fig:R2AlphaDiagramKwI}
		Phase diagram showing the local stability of the $1$-splay states for system~\eqref{eq:KwI_2order} in dependence of the phase-lag parameter $\alpha$, the second moment of the order parameter $R_2(\bm{\vartheta})$, and damping  $\gamma$. In panel (a), the surface separating stable from unstable regimes is depicted in green. The surface corresponds to condition (ii) for Eq.~\eqref{eq:KwISurfaceEq}. In panels (b)-(e), sections for fixed values of $\gamma=0.1,0.5,1,3$ are depicted, respectively. Stable regions are shaded.}
\end{figure}

In Figure~\ref{fig:1splays23456}, we illustrate different $1$-splay states and their corresponding second order parameter $R_2$.
\begin{figure}
	\centering
	\includegraphics{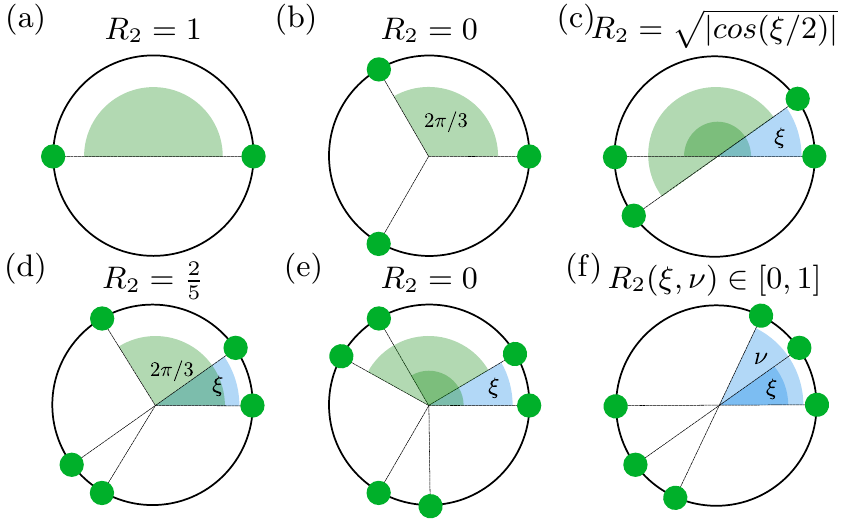}
	\caption{
		\label{fig:1splays23456}
		Illustrations of $1$-splay states ($R_1=0$) for system sizes of (a) $N=2$, (b) $N=3$, (c) $N=4$, (d) $N=5$ and (e),(f) $N=6$. The phases of each phase oscillator are represented on the unit circle by $\vartheta_i\mapsto \exp(\mathrm{i}\vartheta_i)$. The green and blues angles depict fixed and parametrized (variable) phase relations, respectively.}
\end{figure}

%----------------------------------------
% subsection: Adaptive phase oscillator models
%----------------------------------------
\section{Adaptive phase oscillator models\label{sec:POadap}}

In this section, we study the linear stability of generalized splay states in the following general class of coupled phase oscillators with adaptation
\begin{align}\label{eq:POAdapivePhi}
	\frac{\mathrm{d}}{\mathrm{d}t}\bm{\phi} & = \omega \bm{1} + F(\bm{\phi},\bm{\kappa}),\\
	\frac{\mathrm{d}}{\mathrm{d}t}\bm{\kappa} & = -\epsilon\bm{\kappa} + G(\bm{\phi}),
	\label{eq:POAdapiveKappa}
\end{align}
where $\omega$ is the common natural frequency of the phase oscillators,  $F(\bm{\phi},\bm{\kappa})=(f_1(\bm{\phi},\bm{\kappa}),\dots,f_N(\bm{\phi},\bm{\kappa}))^T$ is the coupling vector field with coupling functions $f_i$. The adaptivity variables are given by $\bm{\kappa}= (\kappa_1,\dots,\kappa_K)^T\in \mathbb{R}^K$. Their dynamics is determined by the  dissipation parameter $\epsilon$ and the adaptation vector field $G(\bm{\phi})=(g_1(\bm{\phi}),\dots,g_K(\bm{\phi}))^T$ and adaptation functions $g_l$, $l=1,\dots,K$.

Due to the existence of adaptivity variables $\bm{\kappa}$, we have to generalize the definition of a splay state. A phase-locked state with $\phi_i(t) = \Omega t + \vartheta_i$ and $\kappa_l(\bm{\vartheta})=g_l(\bm{\vartheta})/\epsilon$ ($l=1,\dots,K$) is said to form a generalized $m$-splay state if $R_m(\bm{\vartheta})=0$. In order guarantee the existence of $m$-splay states, we extent the Hypothesis~\ref{hyp:mSplayExists} accordingly. In particular, we assume 
\begin{hypothesis}\label{hyp:mSplayAdapt}
	For all $\bm{\vartheta}\in SM_m$,
	system of coupled phase oscillators 
	\eqref{eq:POAdapivePhi}--\eqref{eq:POAdapiveKappa}
	 possesses $m$-splay states $\bm{\phi}(t)=\Omega t +\bm{\vartheta}$, $\kappa_l(\bm{\vartheta})=g_l(\bm{\vartheta})/\epsilon$
	 with collective frequency $\Omega$.
\end{hypothesis}
Additionally, we assume the phase-shift symmetry.
\begin{hypothesis}\label{hyp:phaseinv-adapt}
	For any $\psi\in \mathbb R$, the nonlinearities $F$ and $G$ satisfy
	$F(\bm{\phi}+\psi\bm{1},\bm{\kappa}) = F(\bm{\phi},\bm{\kappa})$ and $G(\bm{\phi}+\psi\bm{1}) = G(\bm{\phi})$. This implies that the corresponding system \eqref{eq:POAdapivePhi}--\eqref{eq:POAdapiveKappa} is equivariant with respect to the phase-shift transformation. 
\end{hypothesis}

\subsection*{Stability of $m$-splay states}
To study the linear stability of $m$-splay states, we consider the variational equation of~\eqref{eq:POAdapivePhi}--\eqref{eq:POAdapiveKappa} around an arbitrary $m$-splay state, which reads
\begin{align}\label{eq:POAdaptive_variEq}
	\frac{\mathrm{d}}{\mathrm{d}t}
	\begin{pmatrix}
		\delta\bm{\phi} \\
		\delta\bm{\kappa}
	\end{pmatrix} =
	\begin{pmatrix}
		L & B \\
		C & -\epsilon\mathbb{I}_M
	\end{pmatrix}
	\begin{pmatrix}
		\delta\bm{\phi} \\
		\delta\bm{\kappa}
	\end{pmatrix}
	= J
	\begin{pmatrix}
		\delta\bm{\phi} \\
		\delta\bm{\kappa}
	\end{pmatrix},
\end{align}
where $J=J(\bm{\vartheta})$ denotes the Jacobian. 
Due to the phase-shift symmetry,   $(\delta\bm{\phi}^T,\delta\bm{\kappa}^T) = (\bm{1}^T,0,\dots,0)$ is an eigenvector of $J$ with zero eigenvalue and hence $L\bm{1} = 0$.
The entries of the $N\times K$ matrix $B=B(\bm{\vartheta})$, the $K\times K$ matrix $C=C(\bm{\vartheta})$, and the $N\times N$ matrix $L=L(\bm{\vartheta})$ are given as follows.  The nondiagonal entries of $L$ are given by
\begin{align}\label{eq:splayLAdaptiveNonDiag}
	l_{ij} = \frac{\partial f_i}{\partial \phi_j}(\bm{\vartheta},\bm{\kappa}(\bm{\vartheta})).
\end{align}
The diagonal elements $l_{ii}$ are given such that the row sums vanish, i.e., $\sum_{j=1}^N l_{ij}=0$ for all $i=1,\dots,N$. We have
\begin{align}\label{eq:splayLAdaptiveDiag}
	l_{ii}= -\sum_{j=1,j\ne i}^N\frac{\partial f_i}{\partial \phi_j}(\bm{\vartheta},\bm{\kappa}(\bm{\vartheta})).
\end{align}
The entries of $B$ and $C$ are given by
\begin{align}\label{eq:splayBAdaptive}
	b_{il} = \frac{\partial f_i}{\partial \kappa_l}(\bm{\vartheta},\bm{\kappa}(\bm{\vartheta}))
\end{align}
and
\begin{align}\label{eq:splayCAdaptive}
	c_{li} = \frac{\partial g_l}{\partial \phi_i}(\bm{\vartheta}),
\end{align}
respectively.

In order to understand the linear stability of the phase-locked states, we have to determine the eigenvalues of the Jacobian matrix $J$. The following Lemma provides a useful tool to find these eigenvalues. In the following, we use the superscript $H$ to indicate the Hermitian conjugate.
\begin{lem}~\label{lem:POAdaptiveSchurRed}
	Let $L$, $B$, and $C$ be any complex $N\times N$, $N\times K$ and $K\times N$ matrices, respectively. For the $N+K$ eigenvalues of the matrix $J=\begin{pmatrix}
		L & B \\
		C & -\epsilon\mathbb{I}_K
	\end{pmatrix}$, the following statements hold true.\\
	(i) If $K < N$, the eigenvalues of $J$ are given by the solutions of
	\begin{align*}
		(\epsilon+\mu)^{K-N}
		\det\begin{pmatrix}
			(\mu+\epsilon) \left(\mu\mathbb{I}_N-L\right) - BC
		\end{pmatrix} = 0.
	\end{align*}	
	(ii) If $K \ge N$, $J$ possesses $K-N$ eigenvalues $-\epsilon$. The $2N$ remaining eigenvalues are given by 
	\begin{align*}
		\det\begin{pmatrix}
			(\mu+\epsilon) \left(\mu\mathbb{I}_N-L\right) - BC
		\end{pmatrix} = 0.
	\end{align*}
\end{lem}
\begin{proof}
	Using the Schur complement~\cite{BOY04}, we write
	\begin{multline*}
		\det\begin{pmatrix}
			L-\mu\mathbb{I}_N & B \\
			C & -(\epsilon+\mu)\mathbb{I}_K
		\end{pmatrix} = \\
		-(\epsilon+\mu)^K
		\det\begin{pmatrix}
			\left(L-\mu\mathbb{I}_N\right) +(\epsilon+\mu)^{-1} BC
		\end{pmatrix}
		= \\
		-(\epsilon+\mu)^{K-N}
		\det\begin{pmatrix}
			(\mu+\epsilon) \left(\mu\mathbb{I}_N-L\right) - BC
		\end{pmatrix}.
	\end{multline*}
	(i) Suppose $K<N$, then the $K\times N$ matrix $C$ has at least an $N-K$-dimensional kernel. Hence, there exists an $N-K$ dimensional vector space $V$ such that for all $v\in V$
	\begin{align*}
		\begin{pmatrix}
			(\mu+\epsilon) \left(\mu\mathbb{I}_N-L\right) - BC
		\end{pmatrix}v = 0
	\end{align*}
	for $\mu=-\epsilon$. Therefore, the polynomial
	\begin{align*}
		\det\begin{pmatrix}
			(\mu+\epsilon) \left(\mu\mathbb{I}_N-L\right) - BC
		\end{pmatrix} = 0
	\end{align*} possesses at least $N-K$ roots $-\epsilon$.\\
	(ii) Suppose $K\ge N$, we find 
	\begin{multline*}
		\det\left(J-\mu\mathbb{I}_{N+K}\right)=\\-(\epsilon+\lambda)^{K-N}
		\det\begin{pmatrix}
			(\mu+\epsilon) \left(\mu\mathbb{I}_N-L\right) - BC
		\end{pmatrix}
	\end{multline*}
	and $J$ possesses $K-N$ eigenvalues $-\epsilon$. All other eigenvalues are given by the coupled set of $N$ quadratic equations $\det\begin{pmatrix}
		(\mu+\epsilon) \left(\mu\mathbb{I}_N-L\right) - BC
	\end{pmatrix}=0$.
\end{proof}

With this Lemma and Lemma~\ref{lem:splayMatrixCharPolGeneral}, the following conditions for the local stability of the $m$-splay states are derived. Note that $L$ and $BC$ do not necessarily commute.
\begin{prop}\label{prop:splayStateLinStability_POAdaptivemodel}
	Suppose $J$ is the Jacobian of~\eqref{eq:POAdaptive_variEq} and $L$, $B$, and $C$ possess the entries as given in~\eqref{eq:splayLAdaptiveNonDiag}--\eqref{eq:splayLAdaptiveDiag},~\eqref{eq:splayBAdaptive}, and~\eqref{eq:splayCAdaptive}, respectively, where $(\bm{\vartheta},\bm{\kappa}(\bm{\vartheta}))$ corresponds to an $m$-splay state which solves~\eqref{eq:POAdapivePhi}--\eqref{eq:POAdapiveKappa}. Let us further write $\tilde{L}=BC$. Then the $m$-splay state is linearly stable if and only if $\epsilon>0$ and for all solutions $\mu_{1,2,3,4}$ of the quartic equation
	\begin{widetext}
		\begin{multline}\label{eq:POAdaptive_SplayEigenQuartic}
			\mu^4 + \left(2\epsilon-\Tr(L)\right) \mu^3 + \left(\epsilon^2 - 2\epsilon\Tr(L) +\frac{\Tr(L)^2-\Tr(L^2)}{2}-\Tr(\tilde{L})\right)\mu^2 \\
			+\left(\Tr(L)\Tr(\tilde{L})-\Tr(L\tilde{L}) + \epsilon(\Tr(L)^2-Tr(L^2)-\Tr(\tilde{L}))-\epsilon^2\Tr(L)\right)\mu \\
			+ \frac{\Tr(\tilde{L})^2-\Tr(\tilde{L}^2)+2\epsilon(\Tr(L)\Tr(\tilde{L})-\Tr(L\tilde{L}))+\epsilon^2(\Tr(L)^2-\Tr(L^2))}{2}=0,
		\end{multline}
	\end{widetext}
	we have $\mathrm{Re}(\mu_{1,2,3,4})<0$.
\end{prop}
The proof of this Proposition is given in Appendix~\ref{sec:SplayStabPOAdaptive}. With this result, the stability of an $m$-splay state depends on the traces of $L$, $L^2$, $\tilde{L}=BC$, $\tilde{L}^2$, and $L\tilde{L}$ explicitly. We note that, as it has been shown in Ref.~\onlinecite{BER21a}, the phase oscillator models with inertia are a subclass of phase oscillator models with adaptivity. In particular, considering $L=0$  in~\eqref{eq:POAdaptive_SplayEigenQuartic} completely resembles the finding for phase oscillator models with inertia in~\eqref{eq:POwISplayStabQuartic}.

\section{Geometric perspective}\label{sec:geometrical}

This short section aims at giving a qualitative geometric view on the obtained results. In fact, the $m$-splay states are a particular class of incoherent states satisfying a special but important condition $Z_m=0$, i.e., the $m$th order parameter vanishes. Due to surprisingly frequently arising symmetries or special coupling configurations in dynamical networks (Kuramoto-Sakaguchi for example), these states appear and  form high-dimensional manifolds $SM_m$, with the dimension $D$ =  (dimension of the phase space) - 2, i.e., the number of real-valued conditions from $Z_m=0$. Due to the high dimensionality, such states and their stable/unstable manifolds play a crucial role in the global dynamics.

Here we show that the manifold $SM_m$ of the splay states is foliated by the two parameters $\mathrm{Tr}\,(L)$ and $\mathrm{Tr}\,(L^2)$, such that each $(D-2)$-dimensional sheet of this foliation has the same local stability properties. One part of this foliation can be stable, another part is unstable, and the corresponding eigenvalues are given explicitly. Moving along the manifold (changing $\mathrm{Tr}\,(L)$ and $\mathrm{Tr}\,(L^2)$) one can observe classical local bifurcations. 

Our generalizations on phase oscillators with inertia or with adaptation show that the above general geometric picture is preserved, but with different dimensions and some more parameters of the foliation. For adaptive networks, for example, the foliation parameters are the traces of $L$, $L^2$, $\tilde{L}=BC$, $\tilde{L}^2$, and $L\tilde{L}$.

%----------------------------------------
% section: Conclusion
%----------------------------------------
\section{Conclusions}\label{sec:conclusion}

In this article, we have introduced generalized $m$-splay states as a concept for incoherent phase-locked states in ensembles of a finite number of phase oscillators. We have provided a description of their shape and illustrated these states for a small number of oscillators in Section~\ref{sec:model}. Additionally, we have considered each splay state as part of an $N-2$ dimensional manifold called the splay manifold. In Section~\ref{sec:1Cl_stab}, we have described the local dynamical properties of generalized $m$-splay states and have given explicit stability conditions for their stability. Here, we have identified two specific properties of the Jacobian matrix $L$ to be of relevance for the stability. In particular, we have shown that the traces of $L$ and $L^2$ describe the stability for any splay state. 

In order to illustrate these abstract results from Section~\ref{sec:splayStability}, we have applied our findings in Section~\ref{sec:KSmodel} to the Kuramoto-Sakaguchi model that possesses $1$-splay states. We have found that the stability for all splay states is determined by the phase lag parameter $\alpha$ alone. However, it is notable that the local dynamics around each $1$-splay state is determined by the second moment of the order parameter $R_2$. Depending on $R_2$, a $1$-splay state is either a node or a focus.

In Section~\ref{sec:POwI}, the results have been transferred to phase oscillator models with inertia. In Section~\ref{sec:POwISplay}, we have generalized the findings for the stability of generalized splay states and have demonstrated the stability in dependence on the damping constant $\gamma$ and the traces of $L$ and $L^2$. We have further described analytically the two-dimensional surfaces that separate stable regions from unstable regions in $(\Tr(L),\Tr(L^2),\gamma)$-space. As before, we have applied the general results to a specific model. Here, we have considered the Kuramoto-Sakaguchi model with inertia which possesses $1$-splay states. Due to our previous findings, we have derived the shape of the two-dimensional surface explicitly that separates stable from unstable regions in $(\alpha,R_2,\gamma)$-space. In contrast to the pure Kuramoto-Sakaguchi model, the stability of the $1$-splay states depends explicitly on $R_2$ for the model with inertia. Thus, the splay manifold consists of stable and unstable regions. The phase oscillator model with inertia that has been considered in this article can also be interpreted as a phase oscillator model with adaptivity~\cite{BER21a}.

As the last part of this article, we have shown the generic stability condition of any $m$-splay state for a very generic class of adaptive phase oscillator models. Here, we have observed that the stability is not determined by the traces of $L$ and $L^2$ alone. It turns out that yet another Laplacian matrix $\tilde{L}$ describing the interaction of the phases with the adaptive variables is needed to understand the stability properties. Hence, the bifurcation scenarios can be more complex.

In summary, in this article we have developed a general framework to study the local dynamical features of generalized splay states. These states generalize certain concepts of incoherent states as they have been studied previously~\cite{STR91}. In contrast to Ref.~\onlinecite{STR91}, the findings in this article are valid for ensembles of finite size as well. For the particular class of Kuramo-Sakaguchi models, we have also pinpointed the important characteristics that describe the local dynamics transverse to the splay manifold even beyond pure phase oscillator models. Due to the intimate relation between partial integrability and the splay manifold as proposed by the Watanabe-Strogatz approach~\cite{WAT94}, we believe that the present findings provide important insights for future development of generalized dimension reduction techniques.

In the field of chimera states, splay states play an important role for both transient and asymptotic dynamics. Multiple coexistence of splay states with chimera states gives typically rise to riddled and intermingled basins of attraction causing the extreme sensitivity and unpredictability of the global network dynamics\cite{BRE21}. 
 
Another field for application of our results lies in the research on epileptic seizures. It was shown that a drop of the degree of synchronization may occur just before the onset of a seizure~\cite{MOR03a,AND16,GER20}. In particular this drop of synchronization, where the order parameter tends to zero, hints at the dynamical importance of splay states (incoherence) for the emergence of seizures. Moreover, modern approaches to treat tinnitus~\cite{TAS12,TAS12a} and Parkinson's disease~\cite{TAS03a} make active use of incoherent states. In this regard our findings may offer new insights since these methods essentially rely on the stability of incoherent states.

\begin{acknowledgments}
	We dedicate this paper to the memory of Vadim S. Anishchenko. This work was supported by the German Research Foundation DFG, Project Nos. 411803875 and 440145547. 
\end{acknowledgments}

\section*{Data Availability Statement}
The data that supports the findings of this study are available within the article.
%----------------------------------------
% section: Appendix
%----------------------------------------
\appendix
\section{Proof of Lemma~\ref{lem:splayMatrixCharPolGeneral}\label{sec:proofSplayLem}}
	Let us prove the result by complete induction. Consider the case $N=2$. By direct calculation we find that the statement of the Lemma holds true. Now assume the result holds for any $N$. Consider the characteristic polynomial of the following $(N+1)\times(N+1)$ matrix and assume that it possesses $N-1$ roots at zero:
	\begin{align*}
		L_{(N+1)(N+1)} = \begin{pmatrix}
			l_{(N+1)(N+1)} & \begin{matrix}
				l_{(N+1)1} & \cdots & l_{(N+1)N}
			\end{matrix} \\
			\begin{matrix}
				l_{1(N+1)} \\
				\vdots \\
				l_{N(N+1)}
			\end{matrix} &
			L_N 
		\end{pmatrix},
	\end{align*}
	where
	\begin{align*}
		L_N = \begin{pmatrix}
			l_{11} & \cdots & l_{1N} \\
			\vdots & \ddots & \vdots \\
			l_{N1} & \cdots & l_{NN}
		\end{pmatrix}.
	\end{align*}
	We use the Laplace expansion of $\det(L_{(N+1)(N+1)}-\lambda\mathbb{I}_{(N+1)})$ with respect to the first column. We get
	\begin{multline*}
		\det(L_{(N+1)(N+1)}-\lambda\mathbb{I}_{(N+1)}) = \\
		\left(l_{(N+1)(N+1)}-\lambda\right)\det\left(L_N-\lambda \mathbb{I}_N\right) + \\
		+\sum_{i=1}^N(-1)^i l_{i(N+1)}\det \hat{L}_{N,i}
	\end{multline*}
	with $\hat{L}_{N,i}$ given by
%	\begin{widetext}
		\begin{multline*}
\begin{pmatrix}
				l_{(N+1)1} & l_{(N+1)2} & \cdots & \cdots & \cdots & l_{(N+1)N} \\
				l_{11}-\lambda & l_{12} & \cdots & \cdots & \cdots & l_{1N}  \\
				\vdots &  & \ddots  & & & \vdots\\
				l_{(i-1)1} & \cdots & \cdots & l_{(i-1)(i-1)}-\lambda & \cdots  & l_{(i-1)N} \\
				l_{(i+1)1} & \cdots & \cdots & l_{(i+1)(i+1)}-\lambda & \cdots & l_{(i+1)N} \\
				\vdots & & & & \ddots & \vdots \\ 
				l_{N1} & \cdots & \cdots & \cdots & \cdots & l_{NN}-\lambda
			\end{pmatrix}.
		\end{multline*}
%	\end{widetext}
	Remember, by assumption, while evaluating the characteristic polynomial $p_{(N+1)}(\lambda)$ of $L_{(N+1)(N+1)}$, we only have to consider contributions to the coefficients $a_{N+1}$ and $a_{N}$. Note that $\det \hat{L}_{N,i}$ is already a polynomial of degree $N$, thus it can contribute to $a_{N}$ only. Consider an additional Laplacian expansion of $\det \hat{L}_{N,i}$ with respect to the first row. Let $(\hat{L}_{N,i})_i$ be the matrix where we cut off the first row and the $i$th column of $\hat{L}_{N,i}$. We find that the term $-(-1)^i l_{({N+1})i}\det\hat{L}_{N,i,i}$ of the Laplacian expansion of $\det \hat{L}_{N,i}$ contributes to a polynomial of degree $N$. Any other term results in a polynomial in $\lambda$ with degree lower than $N$. Apply the induction ansatz that the statement in the Lemma holds for any $K\le N$, we find
	\begin{multline*}
		\det(L_{(N+1)(N+1)}-\lambda\mathbb{I}_{(N+1)}) = \\ -(-1)^N\lambda^{(N-1)}\left(\lambda^2+\bar{a}_{(N-1)}\lambda+\bar{a}_{(N-2)}\right) + \\
		+ (-1)^Nl_{(N+1)(N+1)}\lambda^{(N-1)}\left(\lambda+\bar{a}_{(N-1)}\right) - \\
		(-1)^{N-1}\sum_{i=1}^N l_{i(N+1)} l_{({N+1})i}\lambda^{N-1}
	\end{multline*}
	where $\bar{a}_k$ are the coefficients of $p(L_N,\lambda)$. Reorganizing the last equation yields the proof.
%%%

\section{Proof of Proposition~\ref{prop:splayStateLinStability_POAdaptivemodel}\label{sec:SplayStabPOAdaptive}}	
Due to Lemma~\ref{lem:POAdaptiveSchurRed}, the eigenvalues of the Jacobian $J$ are determined by the solutions of 
\begin{align*}
	\det\begin{pmatrix}
		(\mu+\epsilon) \left(\mu\mathbb{I}_N-L\right) - BC
	\end{pmatrix} = 0.
\end{align*}
By assumption, the $m$-splay states form an ($N-2$)-dimensional manifold $SM_m$. Consider the $N-2$ perturbation directions along this family $\delta\hat{\bm{\phi}}=(\delta\bm{\phi}^T,0,\dots,0)^T$, where the components are given by $\sum_{j=1}^N e^{\mathrm{i}m\vartheta_j}\delta\phi_j = 0$. From $J\delta\hat{\bm{\phi}}=0$, we get 
$L\delta\bm{\phi}=0$ and $BC\delta\bm{\phi}=B0=0$. Hence,
 the matrix $\bar L(\mu)=(\mu+\epsilon)L + BC $ needs to have at least $N-2$ zero eigenvalues for any $\mu$ and thus~Lemma~\ref{lem:splayMatrixCharPolGeneral} applies. The eigenvalues of $\bar L$ are $0$ with algebraic multiplicity $N-2$ and the solutions of $\lambda^2+a_{(N-1)}\lambda+a_{(N-2)}=0$
where the coefficients read
\begin{align*}
	a_{(N-1)} &= -(\mu+\epsilon)\mathrm{Tr}(L)-\mathrm{Tr}(BC),
\end{align*}
and 
\begin{multline*}
	2a_{(N-2)} %= (\mu+\epsilon)^2\mathrm{Tr}(L)^2 + 2(\mu+\epsilon)\mathrm{Tr}(L)\mathrm{Tr}(BC) + \mathrm{Tr}(BC)^2 \\
	%- \mathrm{Tr}((\mu+\epsilon)^2 L^2+(\mu+\epsilon)LBC+(\mu+\epsilon)BCL+(BC)^2) \\
	= (\mu+\epsilon)^2(\mathrm{Tr}(L)^2- \mathrm{Tr}(L^2)) \\
	+ 2(\mu+\epsilon)(\mathrm{Tr}(L)\mathrm{Tr}(BC)-\mathrm{Tr}(LBC)) \\+ \mathrm{Tr}(BC)^2 - \mathrm{Tr}((BC)^2),
\end{multline*}
where we have used well-known relations for the trace. Knowing the eigenvalues of $\bar L$, we can introduce a Hermitian transformation $Q$ such that $Q^H \bar L Q$ is upper triangular, see Schur form of a matrix in Ref.~\onlinecite{LIE15} for a proof of the existence of $Q$. With this, the polynomial equation
\begin{align*}
	\det\begin{pmatrix}
		(\mu+\epsilon)\mu\mathbb{I}_N - Q^H \bar L(\mu) Q
	\end{pmatrix} = 0
\end{align*}
possesses $N-2$ solutions $\mu=0$ and correspondingly $N-2$ solutions $\mu=-\epsilon$. The four other solutions are given by the two quadratic equations $\mu^2+\epsilon\mu-\lambda_{1,2}(\mu)=0$ where $\lambda_{1,2}$ solve $\lambda^2+a_{(N-1)}\lambda+a_{(N-2)}=0$ with $a_{(N-1)}$ and $a_{(N-2)}$ as above. The quartic form in~\eqref{eq:POAdaptive_SplayEigenQuartic} follows directly by elementary algebraic transformations.

%\section{AKS}
%The nondiagonal entries of the $N\times N$ matrix $L$ are given by
%\begin{align}\label{eq:splayJacobianAKS}
%	l_{ij} = \frac{\sigma}{2N}\left(\sin(\alpha-\beta) - \sin(2(\vartheta_i-\vartheta_j)+\alpha+\beta)\right)
%\end{align}
%The diagonal elements $l_{ii}$, however, are given such that the row sums vanish, i.e., $\sum_{j=1}^N l_{ij}=0$ for all $i=1,\dots,N$. Hence, we find
%\begin{align}\label{eq:splayJacobianDiagAKS}
%	l_{ii}= -\frac{\sigma}{2}\sin(\alpha-\beta)-\frac{\sigma}{N}\cos\alpha\sin\beta
%\end{align}
%
%The nondiagonal entries of the $N\times N$ matrix $\tilde{L}=BC$ are given by
%\begin{align}\label{eq:splayJacobianTildeAKS}
%	\tilde{l}_{ij} = -\frac{\epsilon´\sigma}{2N}\left(\sin(\alpha-\beta) + \sin(2(\vartheta_i-\vartheta_j)+\alpha+\beta)\right)
%\end{align}
%The diagonal elements $\tilde{l}_{ii}$, however, are given such that the row sums vanish, i.e., $\sum_{j=1}^N \tilde{l}_{ij}=0$ for all $i=1,\dots,N$. Hence, we find
%\begin{align}\label{eq:splayJacobianDiagTildeAKS}
%	\tilde{l}_{ii}= \frac{\epsilon\sigma}{2}\sin(\alpha-\beta)-\frac{\epsilon\sigma}{N}\sin\alpha\cos\beta
%\end{align}
%----------------------------------------
% section: Bibliography
%----------------------------------------
\section*{References}
%\bibliographystyle{prwithtitle}
%\bibliography{ref}

\end{document}